\documentclass[12pt]{amsart}
\usepackage{amsmath}
\usepackage{amsfonts}
\usepackage{amscd}
\usepackage{amsmath}
\usepackage{latexsym}
\usepackage{amssymb}
\usepackage{amsthm}
\usepackage[pdftex]{hyperref}
\hypersetup{bookmarksopen,linkcolor=blue,citecolor=magenta,colorlinks}

\numberwithin{equation}{section}

\newtheorem{theorem}[subsection]{Theorem}
\newtheorem{lemma}[subsection]{Lemma}
\newtheorem{proposition}[subsection]{Proposition}

\newtheorem{remark}[subsection]{Remark}

\theoremstyle{definition}

\theoremstyle{remark}

\newcommand{\F}{{\bf F}}

\newcommand{\tr}{\mathop{\rm tr}}
\newcommand{\Tr}{\mathop{\rm Tr}}
\newcommand{\Norm}{{\mathop{\rm N}}}
\newcommand{\LT}{\mathop{\rm LT}}

\newcommand{\barN}{\overline{N}}

\newcommand{\N}{\mathbb{N}}

\newcommand{\A}{\mathcal{A}}
\newcommand{\B}{\mathcal{B}}

\newcommand{\set}[1]{\{#1\}}

\title[Second Main Theorem for $C_2$]{The Second Main Theorem Vector for the modular regular representation of $C_2$.}

\author[Campbell]{H E A Campbell}
\address{Department of Mathematics \& Statistics Department \\
\hfil\break\indent University of New Brunswick \\ Fredericton NB E3B 5A3 \\Canada }
\email{eddy@unb.ca}
\author[Wehlau]{D L Wehlau}
\address{Department of Mathematics and Computer Science \\ \hfil\break\indent
        Royal Military College \\ King\-ston, Ontario, Canada \\ K7K 5L0
       }\email{wehlau@rmc.ca}

\date{\today}


\dedicatory{}

\thanks{This research is supported in part by the Natural Sciences and
Engineering Research Council of Canada. The symbolic computation
language MAGMA (http://magma.maths.usyd.edu.au/) was very helpful.}

\begin{document}
\begin{abstract}
We study the ring of invariants for a finite dimensional representation $V$ of the group $C_2$ of order 2 in characteristic $2$.
Let $\sigma$ denote a generator of $C_2$ and $\{x_1,y_1 \dots, x_m,y_m\}$ a basis of $V^*$.
Then $\sigma(x_i) = x_i$, and $\sigma(y_i) = y_i + x_i$.

To our knowledge, this ring (for any prime $p$) was first studied by David Richman \cite{Richman-vectinvaoverfini:90} in 1990.  He gave a {\em first} main theorem for $(V_2, C_2)$, that is, he proved that the ring of invariants when $p=2$ is generated by
    $$
        \set{x_i, N_i = y_i^2 + x_iy_i, \tr(A) \bigm| 2 \le |A| \le m}\,,
    $$
where $A \subset \set{0,1}^m$, $y^A = y_1^{a_1} y_2^{a_2} \cdots y_m^{a_m}$ and
    $$
        \tr(A) = y^A + (y_1+x_1)^{a_1}(y_2+x_2)^{a_2} \cdots (y_m+x_m)^{a_m}\,.
    $$
In this paper, we prove the {\em second} main theorem for $(V_2, C_2)$, that is, we show that all relations between these generators are generated by relations of type I
    $$
        \sum_{I \subset A } x^I \tr(A-I) = 0\,,
    $$
and of type II
    \begin{align*}
        \tr(A) \tr(B) &= \sum_{L < I} x^{I-L} N^L \tr(I-L+J+K) \\
            &\quad + N^I \sum_{L < J} x^{J-L}\tr(L+K)\,,
    \end{align*}
for all $m$.  We also derive relations of type~III which are simpler and can be used in place of the relations of type~II.
\end{abstract}

\maketitle

\section{Introduction}
 Let $C_p$ be the cyclic group of order $p$ and $V_2$ its unique indecomposable $2$-dimensional representation $V_2$ over a field $\F$ of characteristic $p$.  We take $\set{x,y}$ to be a basis for the hom-dual $V_2^*$ of $V_2$ and we assume the action of $C_p$ on $V_2^*$ to be $\sigma(x) = x$ and $\sigma(y) = y + x$.  Consider the diagonal action of $C_p$ on
    $$
        m\, V_2 = \underbrace{V_2 \oplus V_2 \oplus \cdots \oplus V_2}_m\,.
    $$
and its dual.  Taking a basis for $mV_2^*$ to be $\set{x_1, y_1, \dots, x_m,y_m}$ we obtain $\sigma(x_i) = x_i$ and $\sigma(y_i) = y_i + x_i$.

We are interested in the symmetric algebra of the dual, which we denote $\F[m\, V_2]$ and the corresponding ring of ``vector'' invariants $\F[m\, V_2]^{C_p}$.  This ring of invariants consists of all elements of $\F[m\, V_2]$ fixed point-wise by $\sigma$.
The terminology ``vector invariants'' comes to us from H. Weyl in his book {\em Classical Groups} see \cite{Weyl-clasgrou:97}.
Given a representation of a group $G$ on a vector space over a field $\F$ he refers to theorems explicitly describing
generators for $\F[m\, V]^G$ (where $m\, V = V^{\oplus m}$) for all $m$ as {\em first} main theorems.

David Richman's paper \cite{Richman-vectinvaoverfini:90} in 1990 began the study of the vector invariants of $C_p$ acting on its two-dimensional indecomposable representation $V_2$
in characteristic $p$.  He conjectured that
    \begin{align*}
        &x_i, N_i = y^p - x^{p-1}y_i, u_{ij} = x_iy_j - x_j y_i \text{ and }\\
        & \tr(A) ~\bigm|~0 \le a_i \le p-1  \,,
    \end{align*}
generates the ring of invariants, with a proof in the case $p=2$.  Here $\tr(A)$ denotes the ``trace" (or ``transfer'') of $y^A$, namely
    $$
        \tr(A) = \sum_{\sigma \in C_p} \sigma(y^A) = \sum_{i=0}^{p-1} (y_2 + i x_2)^{a_2}(y_2 + i x_2)^{a_2} \cdots (y_m + i x_m)^{a_m}
    $$
Richman's conjecture was proved by Campbell and Hughes in \cite{Campbell+Hughes-Vectinva:97}. Later, Shank and Wehlau \cite{Shank+Wehlau-Compmoduinva:02a} proved that restricting the traces to have degree larger than $2(p-1)$ gave a \emph{minimal} generating set. Campbell, Shank and Wehlau \cite{Campbell+Shank+Wehlau:10} recently gave a new proof of Richman's conjecture for any $p$, which showed that the minimal algebra generating set just described is also a SAGBI basis.  Finally, Wehlau has given another proof \cite{Wehlau-Shank_Conjecture:12}.  It is well-known that these ring of invariants are not Cohen-Macaulay for $m \ge 3$. It is not hard to show that the number $s$ of such minimal generators is
    $$
      p^m - \binom{m+2p-2}{m} +m\binom{m+p-2}{m} + \binom{m}{2} + 2m\,,
    $$

We form a polynomial ring $Q = \F[\xi_1, \xi_2, \dots, \xi_s]$ for $s$ as above and a surjection of algebras $\pi: Q \to \F[m\,V_2]^{C_p}$ by setting $\pi(\xi_i)$ to be the $i^{\text{th}}$ generator. Elements of the kernel of $\pi$ give rise to relations for the ring of invariants. A theorem giving an explicit generators for $\ker(\pi)$ for all $m$ is referred to as a \emph{second main theorem} by Weyl. In this paper, we give two second main theorems when $p = 2$ by exhibiting two minimal generating sets for $\ker(\pi)$.  These sets have
    $$
        2^m - \binom{m}{2} - m - 1 + \binom{2^m-m}{2}
    $$
many generators. We note here that when $p=2$ then $s = 2^m + m - 1$. To our knowledge, this is the first proof of a second main theorem for a finite group in the modular case.

In 1916 Emmy Noether gave a characteristic $0$ bound on the degrees of generators of a ring of invariants for a finite group (the original paper is \cite{Noether-EndlInvaendlGrup:16}, for a modern treatment see \cite{Derksen+Kemper-Compinvatheo:02}).   She showed that the ring of invariants of any characteristic 0 representation of $G$ is generated in degrees less than or equal to $|G|$.
  In contrast, in his paper Richman proved that $\mathbb{K}[m\,V_2]^{C_p}$, for any $p$, required a generator of degree $m(p-1)$.  Thus he demonstrated that Noether's bound does not hold in general for modular groups, those groups whose order is divisible by the characteristic of the underlying field.  In this connection, it is worth noting that Symonds \cite{Symonds:11} recently proved Kemper's conjecture \cite{Kemper:97} that for any modular representation of a finite group, the ring of invariants can be generated in degrees less than or equal to
    $$
        \dim_\F(V)(|G|-1)
    $$
    if $|G|\geq 2$.
Symonds proof uses Castelnuovo-Mumford regularity and builds on his work with Karagueuzian, \cite{Karagueuzian-Symonds:07}.  Moreover, he proves that all the relations between the generators may be found in degrees less than or equal
    $$
        2\dim_\F(V)(|G|-1)
    $$
    if $|G|\geq 2$.
Our theorem implies that for the case we study, $\F_2[m\, V_2]^{C_2}$, this upper bound is sharp. The relation of largest degree is associated to the product
    $$
        \tr(y_1 y_2 \cdots y_m)^2\,.
    $$

By the Hilbert syzygy theorem, the projective dimension of our ring of invariants is less than or equal to $s = 2^m + m - 1$ and
(for $m\geq 3$) bigger than $s-2m = 2^m-m-1$ (the projective dimension of $\F_2[m\, V_2]^{C_2}$ cannot be $2^m-m-1$  since this ring is not Cohen-Macaulay).

\section{Preliminaries}

We refer to the book of Derksen and Kemper, \cite[\S 3.6]{Derksen+Kemper-Compinvatheo:02}.  Suppose we have a commutative graded unitary connected ring $S$ over $\F$ of Krull dimension $n$. Choose a homogeneous system of parameters $\set{z_1, z_2, \dots, z_n}$, and form the polynomial ring $R = \F[z_1, z_2, \dots, z_n]$. Then $S$ is a finitely generated $R$-module on (secondary) $R$-module homogeneous generators $\set{w_1, w_2, \dots, w_s}$.  Let $Q$ denote the polynomial ring $R[t_1, t_2, \dots, t_s]$ and consider the $R$-module homomorphism $\pi: Q \to S$ given by  $\pi(t_i) = w_i$. We have that the kernel of $\pi$ is generated by relations of the two forms:
\begin{itemize}
    \item[] \emph{Linear Relations:} $\qquad\quad\sum_{i=1}^s f_i t_i$;
    \item[] \emph{Quadratic Relations:} $\qquad t_it_j + \sum_{k=1}^s f_k t_k$
\end{itemize}
where all the $f_i, f_k \in R$.

For us, $R = \F[x_1, x_2, \dots, x_m, N_1, N_2, \dots, N_m]$, and, provided $p=2$ we are in the unusual situation that each secondary $R$-module generator $\tr(A)$ (for $|A| \ge 2$) is also a generating invariant.  This fails for $p \ge 3$.  In our situation, therefore, we will write
$$Q = R[\Tr(A)~\mid~A \in \set{0,1}^m,~|A| \ge 2]$$
 and set $\pi(\Tr(A)) = \tr(A)$. It is also natural to grade $Q$ by setting $\deg(x_i) =1$, $\deg(N_i) = 2$ and
 $\deg(\Tr(A)) =  |A| := \sum_{i=1}^m a_i$.  With this grading, $\ker(\pi)$ is homogeneous and, therefore, minimal ideal generating sets are those for which no proper subset generates.
  For the remainder of this paper, we refer to relations as occurring in either $Q$ or the ring of invariants.

\subsection{Notation}
Let $\N=\{0,1,2,\dots\}$.  Let $A=(a_1,a_2,\dots,a_m)\in \N^m$ denote an exponent sequence, for example,
    $$
        x^A = \prod_{s=1}^m x_s^{a_s}, \qquad   y^A = \prod_{s=1}^m y_s^{a_s}\quad\text{ and}\quad N^A = \prod_{s=1}^m N_s^{a_s}\,.
    $$

We write $|A|$ to denote $\sum_{s=1}^m a_s$.  Given $A, B \in \N^m$ we write $A \leq B$ if $a_s \leq b_s$ for all $s=1,2,\dots,m$.

We recalled above the invariant norms $N_s = y_s^2-x_sy_s$.  It is also useful to define $\barN_s = y_s  + x_s = N_s/y_s$.  Given an exponent sequence $A$, we note that
    $$
        \barN^A = \sum _{B \leq A} x^{A-B} y^B = \tr(A) + y^A\,.
    $$

If the sequence $A=(a_1,a_2,\dots,a_m)$ consists of only 0's and 1's, we may consider $A$ as the characteristic vector of  the set $\A := \{1 \leq s \leq m \mid a_s = 1\}$. Similarly we consider the zero/one sequence $B$ of length $m$ as the characteristic vector of the subset $\B$ of $\{1,2,\dots,m\}$.   We use this viewpoint to introduce three more binary operations on zero/one sequences.  Suppose $A, B \in \{0,1\}^m$.  We write $A \cap B$, $A\setminus B$ and $A \cup B$  to denote the zero/one sequences which are the characteristic vectors of the sets $\A \cap \B$, $\A \setminus \B$ and $\A \cup \B$ respectively.
From now on, we will abuse notation and write $A$ for $\A$ and $B$ for $\B$.
For convenience, we define $\Delta_s$ to be the sequence which is $1$ in position $s$ and $0$ everywhere else.

We are going to use the {\em graded reverse lexicographic} order on monomials with
    $$
        y_1 > x_1 > y_2 > x_2 > \cdots y_m > x_m\,.
    $$
We denote the lead term of a polynomial $f$ by $\LT(f)$.  For $A \neq 0$, we denote by $\ell(A)$ the least integer $\ell$ such that $a_\ell = 1$ and $a_i = 0$ for $1 \le i < \ell$, that is, $\ell(A) = \min\set{ \ell \mid a_\ell = 1}$.  We will denote by $A^\prime$ the sequence obtained from $A$ by setting the entry $\ell(A)$ equal to $0$.

\section{A second main theorem}

In this section, we first derive relations of two types, the first of which are linear, and the second of which are quadratic in the sense of Derksen and Kemper. Then we go on to show that these relations minimally generate the ideal of all relations $\ker(\pi)$.

The following lemma is easily shown.
\begin{lemma}\label{trace_formula}
Let $A \in \set{0,1}^m$.  Then
  $$
    \tr(A) = \sum_{L < A} x^{A-L} y^L\,.
  $$
Further, $\LT(\tr(A)) = x_{\ell(A)}y^{A^\prime}$.  \qed
\end{lemma}

Applying the trace to the formula in the above lemma yields the relation $0=\sum_{0 < I < A} x^{A-I}  \tr(y^{I})$.   If $|A| \leq 2$, then this relation is vacuous.  Otherwise we have a meaningful relation which we record in the following

\begin{proposition}\label{type_I}  Suppose $|A| \geq 3$.  Then
  $$
    \sum_{0 < L < A} x^{A-L}\Tr(L) \in \ker(\pi)\,.
  $$
Define $\ell(A) = i$ and $\ell(A^\prime) = j$. Then
    $$
        \LT(\pi(x^{A-L}\Tr(L)) \le x_i x_j y^{(A^\prime)^\prime} \text{ for all } L < A
    $$
with equality if and only if either $A - L = \Delta_i$ or $A-L = \Delta_j$.
\end{proposition}

\begin{proof}
  The first assertion is easily shown as explained above.
  To show the second assertion it suffices to note that $\LT(\tr(x^{A-L}y^L))=x^{A_L}\LT(\tr(y^L)) = x^{A-L}x_{\ell(L)}y^{L'}$.
\end{proof}

We will call the relation in the above proposition the \emph{relation of type~I} associated to the subset $A$ of $\set{0,1}^m$.

\begin{remark}
The relation of type~I corresponding to $\tr(y_1y_2y_3)$, (i.e., to $A=\Delta_1 + \Delta_2 + \Delta_3$) can be used to show that our invariant rings are not Cohen-Macaulay when $m \geq 3$.
\end{remark}

\medskip
Next we describe the \emph{relations of type~II}.

\begin{proposition}\label{type_II}
Let $A, B \in \set{0,1}^m$.  Suppose $|A|\geq 2$, $|B| \geq 2$.  We define $I := A \cap B$, $J := A \setminus B$ and $K := B \setminus A$.
Then
  $$
    \Tr(A) \Tr(B) + \sum_{L < I} x^{I-L} N^L \Tr({I-L+J+K}) + N^I \sum_{L < J} x^{J-L}\Tr(L+K)
  $$
is in $\ker(\pi)$.
\end{proposition}

\begin{remark}
We note that this relation is not symmetric in $A$ and $B$.  The formula above holds but is vacuous if $|A| = 1$.  If $|B| = 1$, then the formula holds, but is a consequence of relations of type~I.
\end{remark}

\begin{proof}
We have
  \begin{align*}
    \tr(A) \tr(B) &= (\barN^A + y^A)(\barN^B + y^B)\\
                             &= \barN^{A+B} + y^A\barN^B + y^B\barN^A + y^{A+B}\\
                             & =\barN^{2I}\barN^{J+K} + y^A\barN^B+ y^B\barN^A + y^{A+B}.
 \end{align*}
Since
  \begin{align*}
      \barN^{2I} &= \prod_{s=1}^m (y_s + x_s)^{2i_s} =  \prod_{s=1}^m (y_s^2 + x_s^2)^{i_s} \\
                 &= \prod_{s=1}^m (x_s \barN_s + N_s)^{i_s} = \sum_{L\leq I} x^{I-L}\barN^{I-L} N^L \\
  \end{align*}
and
  \begin{align*}
    y^A \barN^B &= y^I y^J \barN^I \barN^K = N^I y^J \barN^K = N^I \barN^K \prod_{s=1}^m (x_s+y_s+x_s)^j_s\\
                & = N^I \barN^K \prod_{s=1}^m (x_s+\barN_s)^j_s = N^I \barN^K \sum_{L \leq J} x^{J-L} \barN^L \\
                &= N^I \sum_{L \leq J} x^{J-L} \barN^{K+L}
  \end{align*}
we have
  \begin{align*}
    \tr(A) \tr(B)  &= \sum_{L\leq I} x^{I-L}\barN^{I-L} N^L \barN^{J+K} + N^I \sum_{L \leq J} x^{J-L} \barN^{K+L} \\
                    &\qquad+ y^B\barN^A + y^{A+B}\\
                    & = \sum_{L \leq I} x^{I-L} N^L \barN^{I+J+K-L}+ N^I\sum_{L \leq J} x^{J-L} \barN^{K+L} \\
                    &\qquad+ y^B\barN^A + y^{A+B}\ .
  \end{align*}
This last expression yields
    \begin{align*}
   \tr(A) \tr(B) &= \sum_{L \leq I} x^{I-L} N^L (\tr(I+J+K-L) + y^{I+J+K-L})\\
             &\qquad+ N^I\sum_{L \leq J} x^{J-L} (\tr(K+L) + y^{K+L}) + y^B\barN^A\\
             &\qquad+ y^{A+B}\,,
    \end{align*}
and therefore
    \begin{align*}
    \tr(A) \tr(B) &= \sum_{L < I} x^{I-L} N^L \tr(I+J+K-L) + N^I\tr(J+K) \\
                &\qquad + \sum_{L \leq I} x^{I-L} N^L y^{I+J+K-L} + N^I\sum_{L < J} x^{J-L} \tr(K+L)\\
             &\qquad + N^I \tr(K+J) + N^I\sum_{L \leq J} x^{J-L} y^{K+L}\\
             &\qquad+ y^B\barN^A + y^{A+B}\,,
  \end{align*}
from which we obtain
    \begin{align*}
    \tr(A) \tr(B)&= \sum_{L < I} x^{I-L} N^L \tr(I+J+K-L) + \sum_{L \leq I} x^{I-L} N^L y^{I+J+K-L}\\
             &\qquad+ N^I\sum_{L < J} x^{J-L} \tr(K+L) + N^I\sum_{L \leq J} x^{J-L} y^{K+L}\\
             &\qquad+ y^B\barN^A + y^{A+B}\,
  \end{align*}
Now
  \begin{align*}
      \sum_{L \leq I} x^{I-L} N^L y^{I+J+K-L} &= y^{J+K}(\sum_{L \leq I} x^{I-L} N^L y^{I-L})\\
      &=  y^{J+K}(\sum_{L \leq I}  N^L x^{I-L}y^{I-L})=  y^{J+K}(N+xy)^I \\
      &= y^{J+L} (y^2)^I = y^{2I+J+L} = y^{A+B}\,.
  \end{align*}
Finally
  \begin{align*}
   N^I\sum_{L \leq J} x^{J-L} y^{K+L}&= N^I y^K \sum_{L \leq J} x^{J-L} y^L   \\
   &= y^I \barN^I y^K \barN^J = y^{I+K} \barN^{I+J} = y^B \barN^A
  \end{align*}
completing the proof.
\end{proof}

For each pair $A, B \in \set{0,1}^m$ with $A \neq B$, we have two relations given in Proposition~\ref{type_II} since these relations are not symmetric in $A$ and $B$.  In order to produce a minimal set of generating relations, we need to choose one of these two.

\begin{theorem}[A second main theorem]\label{all_relns_from_I_II}
The ideal of relations among the generators of $\F[m\,V_2]^{C_2}$ is minimally generated by
\begin{itemize}
    \item the $2^m - \binom{m}{2}-m-1$ relations of type~I associated to the subsets $A \subset \set{1,2, \dots m}~\text{with}~ |A| \ge 3$
    \item  $\binom{2^m-m}{2}$ relations of type~II associated to the (unordered) product of any two of the $2^m - m - 1$ traces.
\end{itemize}
\end{theorem}

\begin{proof}
We first prove that the set above generates the full set of relations.  By \cite[\S 3.6]{Derksen+Kemper-Compinvatheo:02}, we need only show that we have an $R$-module basis for all the linear relations.  We proceed by contradiction. Given an element $h \in Q$, we write
    $$
        h = \sum_{I,J,A} \alpha_{I,J,A}x^I N^J \Tr(A), \text{ for }\alpha \in \F
    $$
and define
    $$
        \Gamma(h) = \max_{I,J,A}\set{\LT(x^I N^J\tr(A))~\mid~\alpha_{I,J,A} \ne 0}
    $$
Choose a polynomial $h$ such that $\Gamma(h)$ is minimal among all the linear relations that are not in the $R$-module generated by the type~I relations. Therefore, there exist $I_1, J_1, A_1$ such that
    $$
        \LT(x^{I_1}N^{J_1}\tr{A_1}) = \Gamma(h) = x^{I_1}y^{2J_1}x_{\ell(A_1)}y^{A_1^\prime}\,.
    $$
We will write $\ell(A_1) = a_1$ and $\ell(A_2) = a_2$.  Since $h$ describes a relation, there must exist
 $(I_2, J_2, A_2)\neq (I_1,J_1,A_1)$ such that $x^{I_2}N^{J_2}\tr{A_2}$
$\LT(x^{I_2}N^{J_2}\tr{A_2}) = \Gamma(h)$.
That is,
    $$
        x^{I_1}y^{2J_1}x_{a_1}y^{A_1^\prime} = x^{I_2}y^{J_2}x_{a_2}y^{A_2^\prime}\,,
    $$
so that we have $J_1 = J_2$, $A_1^\prime = A_2^\prime$ and $I_1 + \Delta_{a_1} = I_2 + \Delta_{a_2}$.

We form the type~I relation associated to $A_1 + \Delta_{a_2}$:
    $$
        x^{I_1 - \Delta_{a_2}} N^{J_1}(\sum_{L < A_1 + \Delta_{a_2}}x^{A_1+ \Delta_{a_2}-L}\Tr(L))\,,
    $$
and, as we noted in Proposition \ref{type_I},
    $$
        \Gamma(h) = \Gamma(x^{I_1 - \Delta_{a_2}} N^{J_1}(\sum_{L < A_1 + \Delta_{a_2}}x^{A_1+ \Delta_{a_2}-L}\Tr(L)))
    $$
and occurs as the lead term of just two of the monomials in this sum (after applying $\pi$).  Therefore, we define
    $$
        h^\prime = h - x^{I_1 - \Delta_{a_2}} N^{J_1}(\sum_{L < A_1 + \Delta_{a_2}}x^{A_1+ \Delta_{a_2}-L}\Tr(L))\,,
    $$
and note that $\Gamma(h^\prime) \le \Gamma(h)$.  If $\Gamma(h^\prime) = \Gamma(h)$, then we have reduced the number of terms with this lead term and repeating the process will lead, after a finite number of steps, to a new $h^\prime$ with a smaller lead term.  But $\Gamma(h^\prime) < \Gamma(h)$ contradicts the definition of $h$, as we were required to prove.
\end{proof}

\section{Relations of type III and another second main theorem}

Here we obtain a new set of quadratic relations shorter than the type~II relations of the previous section.  We also show that these relations minimally generate all quadratic relations, thus giving another version of the second main theorem.
\begin{theorem}
We assume that $A,B \in \set{0,1}^m$ and that $|A|, |B| \ge 2$.
    \begin{align*}
        \text{{\bf III (a)}: Suppo}&\text{se } A \cap B = \emptyset. \text{ Let }j \in B \text{ and put }
               B' = B - \Delta_j.\\
            \Tr(A)\Tr(B)  &+  \Tr(A + \Delta_j) \Tr(B') + x_{j} \Tr(A+B')\\
                &\qquad+ x_{j} \Tr(A) \Tr(B') \in \ker(\pi)\\
        \text{{\bf III (b)}: Suppo}&\text{se }|A \cap B| \ge 1\text{ and } B \leq A.\\
              \text{  Let }& i \in A \cap B \text{ and put } B' = B - \Delta_i \text{ and } A' = A - \Delta_i.\\
                \Tr(A) \Tr(B) &+  x_{i} \Tr(A)\Tr(B') + {\Norm}_{i} \Tr(A')\Tr(B') \\
                &\qquad +  x_{i} {\Norm}^{B'}\Tr(J + \Delta_i) \in \ker(\pi) \\
        \text{{\bf III (c)}: Suppo}&\text{se } |A \cap B| \ge 1\text{ and } A \not\leq B \text{ and } B \not\leq A.\\
          \text{ Put }& I = A \cap B, J = A \setminus B \text{ and }K = B \setminus A.\\
             \Tr(A)\Tr(B) &+ \Tr(I+J+K)\Tr(I) + {\Norm}^{I}\Tr(J)\Tr(K) \in \ker(\pi)\,.
    \end{align*}
   \end{theorem}

\begin{proof}
Suppose $C,D \in \{0,1\}^m$ and $C \cap D = \emptyset$.  Then
    $$
        (\tr(C)+y^C)(\tr(D)+y^D) = \tr(C+D) + y^{C+D}\,.
    $$
Hence
    $$
        \tr(C+D) + \tr(C)\tr(D) + y^C\tr(D) + y^D\tr(C) = 0\ .
    $$
We will use this identity repeatedly in the proof.

 \begin{enumerate}
  \item[(a)] Consider the two relations $\tr(B) = y_j\tr(B') + x_j\tr(B') + x_j y^{B'}$ and $\tr(A + \Delta_j) = y_j\tr(A) + x_j\tr(A) + x_j y^{A}$.  Multiplying the first by $\tr(A)$ and the second by $\tr(B')$ and then adding the results yields
\begin{align*}
     \tr(A)\tr(B) + \tr(A + \Delta_j)\tr(B') &= x_j y^{B'}\tr(A) +  x_j y^A \tr(B')\\
     & = x_j \tr(A+B') + x_b\tr(A)\tr(B')
  \end{align*}

 \item[(b)] Multiplying the relation $\tr(B) = y_i\tr(B') + x_i \tr(B') + x_i y^{B'}$ by $\tr(A)$ we have
  \begin{align*}
     \tr(A)&\tr(B)  = y_i\tr(A)\tr(B') + x_i \tr(A)\tr(B') + x_i y^{B'}\tr(A)\\
     &= y_i\tr(A)\tr(B') + x_i \tr(A)\tr(B') \\
     &\quad + x_i y^{B'}\left(\tr(A^\prime + \Delta_i)\tr(B') + y_i y^{A^\prime} \tr(B') + \tr(A^\prime + \Delta_i)y^{B'}\right)\\
     &= y_i\tr(A)\tr(B')  + x_i \tr(A)\tr(B') + x_i \tr(y_i y^J)\Norm^{B'} \\
     &\qquad + x_i y_i y^J y^{B'}\tr(B')\\
     &= \left(y_i\tr(A') + x_i\tr(A') + x_i y^Jy^{B'}\right)\tr(B')  + x_i \tr(A)\tr(B')\\
      &\quad   + x_i \tr(y_i y^J)\Norm^{B'} + x_i y_i y^J y^{B'}\tr(B')\\
     &= {\Norm}_i \tr(A')\tr(B') + x_i \tr(A)\tr(B')
           + x_i \tr(y_i y^J){\Norm}^{B'}
  \end{align*}

 \item[(c)]
  \begin{align*}
     \tr(J)&\tr(K){\Norm}^I  = \tr(J)\tr(K)(y^I\tr(I) + y^{2I})\\
                               & = (y^I\tr(J))(\tr(I)\tr(K) + y^{2I}\tr(K))\\
                               & = (\tr(J+I) + \tr(J)\tr(I)+y^J\tr(I))(\tr(I+K)\\
                               &\qquad +y^K\tr(I)+y^{2I}\tr(K))\\
                               & = \tr(J+I)\tr(I+K) + \tr(I)\left(\tr(J)\tr(I+K) + y^J\tr(I+K) \right)\\
                                &\quad + \tr(I)\left( y^K(\tr(J+I)+\tr(J)\tr(I))+y^J\tr(I)\right)\\
                                & = \tr(A)\tr(B) + \tr(I)(\tr(J)\tr(I+K) \\
                                &\qquad\qquad\qquad + y^J\tr(I+K) + y^Iy^K\tr(J))\\
                                & = \tr(A)\tr(B) + \tr(I)\tr(J+I+K)
  \end{align*}
 \end{enumerate}
 \end{proof}

If either $|A| =1 $ or $|B|=1$ then the relations of type III(a) and type III(b) are valid but are either vacuous or are consequences of type~I relations or are other type~III relations.  The conditions $|A \cap B| \geq 1$ and $A \not \leq B$ and $B \not \leq A$ force $|A| \geq 2$ and $|B|\geq 2$.

\begin{theorem}\label{all_relns_from_I_III}
We form a set of relations by choosing, subject to the following 2 restrictions, one relation of type~III for every pair $A, B \in \set{0,1}^m$ with $|A|, |B| \ge 2$.   If $A$ and $B$ are disjoint, we interchange $A$ and $B$ if necessary to ensure that $|A| \ge |B|$ and then choose a relation of type~III~(a).  If $|A \cap B| \ge 1$ and $ A \subset B$ then we interchange $A$ and $B$ and then choose a relation of type~III~(b).  Any set formed by so choosing one relation for each pair with $|A|, |B| \ge 2$ minimally generates the quadratic relations.   In particular, such a set together with the  Type~I relations minimally generates $\ker \pi$.
\end{theorem}

\begin{proof}
  We show that a set of relations chosen as described, allows us to reduce every product $\tr(A)\tr(B)$ with $|A|, |B| \ge 2$
  to a sum of $R$-linear combinations of the secondary generators $\tr(A)$.

We change the natural grading on $Q$ by setting $\deg(x_i) = \deg(N_i) = 0$ for all $i$, and induct on the resulting degree.
 Suppose we are given $A,B \in \set{0,1}^m$ with $|A|, |B| \ge 2$.

If $B \le A$, then the corresponding relation of type~III~(b) allows us to rewrite $\tr(A)\tr(B)$ as the sum of three terms, each of which is of smaller degree.

If $|A \cap B| \ge 1$ and $A \not\le B$ and $B \not\le A$, then the relation of type~III~(c) rewrites $\tr(A)\tr(B)$ as a sum of two terms, one of which has smaller degree and the other term can be rewritten using a relation of type~III~(b).

If $|A \cap  B| = 0$, then the relation of type~III~(a) produces three terms, two of which have lower degree.  The other term has a factor $\tr(B^\prime)$ with $|B^\prime| < |B|$. We may, therefore, repeatedly apply a relation of type~III~(a) until $|B^\prime| = 1$ at which point $\deg(\tr(B^\prime)) = 0$, finishing the proof.
\end{proof}

\bibliographystyle{amsplain}
\bibliography{C:/LocalTexFiles/bibtex/bib/listkey_master_2012}

\def\cprime{$'$}
\providecommand{\bysame}{\leavevmode\hbox to3em{\hrulefill}\thinspace}
\providecommand{\MR}{\relax\ifhmode\unskip\space\fi MR }
\providecommand{\MRhref}[2]{%
  \href{http://www.ams.org/mathscinet-getitem?mr=#1}{#2}
}
\providecommand{\href}[2]{#2}
\begin{thebibliography}{10}

\bibitem{Campbell+Hughes-Vectinva:97}
H.~E.~A. Campbell and I.~P. Hughes, \emph{Vector invariants of {$U\sb
  2(\mathbf{F}\sb p)$}: a proof of a conjecture of {R}ichman}, Adv. Math.
  \textbf{126} (1997), no.~1, 1--20. \MR{98c:13007}

\bibitem{Campbell+Shank+Wehlau:10}
H.~E.~A. Campbell, R.~J. Shank, and D.~L. Wehlau, \emph{Vector invariants for
  the two-dimensional modular representation of a cyclic group of prime order},
  Adv. Math. \textbf{225} (2010), no.~2, 1069--1094. \MR{2671188 (2011f:13008)}

\bibitem{Derksen+Kemper-Compinvatheo:02}
Harm Derksen and Gregor Kemper, \emph{Computational invariant theory},
  Invariant Theory and Algebraic Transformation Groups, I, Springer-Verlag,
  Berlin, 2002, Encyclopaedia of Mathematical Sciences, 130. \MR{2003g:13004}

\bibitem{Karagueuzian-Symonds:07}
Dikran~B. Karagueuzian and Peter Symonds, \emph{The module structure of a group
  action on a polynomial ring: a finiteness theorem}, J. Amer. Math. Soc.
  \textbf{20} (2007), no.~4, 931--967 (electronic). \MR{2328711}

\bibitem{Kemper:97}
Gregor Kemper, \emph{Hilbert series and degree bounds in invariant theory},
  Algorithmic algebra and number theory ({H}eidelberg, 1997), Springer, Berlin,
  1999, pp.~249--263. \MR{1672054 (2000b:13007)}

\bibitem{Noether-EndlInvaendlGrup:16}
E.~Noether, \emph{Der endlichkeitssatz der invarianten endlicher gruppen},
  Math. Ann. \textbf{77} (1916), 89--92.

\bibitem{Richman-vectinvaoverfini:90}
David~R. Richman, \emph{On vector invariants over finite fields}, Adv. Math.
  \textbf{81} (1990), no.~1, 30--65. \MR{91g:15020}

\bibitem{Shank+Wehlau-Compmoduinva:02a}
R.~James Shank and David~L. Wehlau, \emph{Computing modular invariants of
  {$p$}-groups}, J. Symbolic Comput. \textbf{34} (2002), no.~5, 307--327.
  \MR{2003j:13006}

\bibitem{Symonds:11}
Peter Symonds, \emph{On the {C}astelnuovo-{M}umford regularity of rings of
  polynomial invariants}, Ann. of Math. (2) \textbf{174} (2011), no.~1,
  499--517. \MR{2811606}

\bibitem{Wehlau-Shank_Conjecture:12}
David Wehlau, \emph{Invariants for the modular cyclic group of prime order via
  classical invariant theory}, J.~European~Math.~Soc. \textbf{15} (2013),
  no.~3, 775--803.

\bibitem{Weyl-clasgrou:97}
Hermann Weyl, \emph{The classical groups}, Princeton Landmarks in Mathematics,
  Princeton University Press, Princeton, NJ, 1997. \MR{98k:01049}

\end{thebibliography}

\end{document}